\documentclass[10pt,a4paper]{article}
\usepackage{amsmath,amssymb,amsfonts,amsthm}
\usepackage{graphicx}
\usepackage{float,graphicx,color}
\usepackage[all]{xy}

\newcommand{\rmd}{\mathrm{d}}

\newcommand{\rmH}{\mathrm{H}}
\newcommand{\rmL}{\mathrm{L}}

\newcommand{\bbN}{\mathbb{N}}

\newcommand{\bbR}{\mathbb{R}}

\newcommand{\bbZ}{\mathbb{Z}}

\newcommand{\calF}{\mathcal{F}}

\newcommand{\calH}{\mathcal{H}}

\DeclareMathOperator{\Div}{div}
\renewcommand{\div}{\Div}

\DeclareMathOperator{\grad}{grad}

\newcommand{\bs}{{\scriptscriptstyle \bullet}}

\newcommand{\cleq}{\preccurlyeq}
\newcommand{\cgeq}{\succcurlyeq}
\newcommand{\ceq}{\approx}

\newcommand{\beq}{\begin{equation}}
\newcommand{\eeq}{\end{equation}}

\newcounter{point}[section]
\newcommand{\point}{
\paragraph{(\roman{point})}
\stepcounter{point}
}

\newtheorem{theorem}{Theorem}[section]
\newtheorem{lemma}[theorem]{Lemma}
\newtheorem{corollary}[theorem]{Corollary}
\newtheorem{proposition}[theorem]{Proposition}

\theoremstyle{definition}

\theoremstyle{remark}
\newtheorem{remark}{Remark}[section]

\theoremstyle{plain}

\theoremstyle{definition}

\theoremstyle{remark}

\title{Stability of an upwind Petrov-Galerkin discretization of convection diffusion equations
}

\author{Snorre H. Christiansen\footnote{Department of Mathematics, University of Oslo, PO Box 1053 Blindern, NO 0316 Oslo, Norway}, Tore G. Halvorsen, Torquil M. S\o rensen}
\date{}

\begin{document}

\maketitle

\begin{abstract}
We study a numerical method for convection diffusion equations, in the regime of small viscosity. It can be described as an exponentially fitted conforming Petrov-Galerkin method. We identify norms for which we have both continuity and an inf-sup condition, which are uniform in mesh-width and viscosity, up to a logarithm, as long as the viscosity is smaller than the mesh-width or the crosswind diffusion is smaller than the streamline diffusion. The analysis allows for the formation of a boundary layer.
\end{abstract}

\section*{Introduction}

For many fluid flow problems of relevance to engineering, the convective term, hyperbolic in nature, is moderated by a viscous term, elliptic in nature. For large viscosity, Galerkin finite element methods yield good results.  As the viscosity tends to zero, sharp gradients in the fluid velocity as well as boundary layers will appear. When the characteristic length of boundary layers is smaller than the mesh-width, standard Galerkin methods become unstable. This prominent example of a multiscale problem has motivated a large body of work on stabilised methods. For overviews and introductions we refer to \cite{RooStyTob08} and \cite{Mor10}.

This paper is motivated by a general method, introduced in \cite{Chr13FoCM}, that applies to differential forms on arbitrary meshes. It produces differential complexes of finite element spaces that take into account convection. It treats uniformly differential $k$-forms of all degrees $k$. Thus it fits into the framework of finite element exterior calculus (FEEC) \cite{ArnFalWin10}. More precisely it fits into the framework of finite element systems (FES) \cite{Chr09AWM}, which was designed to accommodate polyhedral meshes and quite general basis functions, yet produce discrete spaces equipped with commuting interpolators. So far we have not provided any analysis of this method in the convection dominated regime, even for scalar problems.

For scalar equations on product grids, the method relates to exponential fitting. Variants of exponential fitting can be traced all the way back to \cite{AllSou55} and \cite{SchGum69}. This method can be analysed quite exhaustively in dimension one, for instance because, in model situations, the discrete solution turns out to interpolate the exact one at vertices. However, already in dimension two, in spite of its naturality, the method is hard to analyse, compared with Galerkin methods at large viscosity.

To be more precise, one defines upwinded finite elements that solve \emph{local} problems related to the adjoint equation. Downwinded elements, on the other hand, locally solve problems related to the original equation. In cases where these problems can be solved explicitly one often obtains exponential functions with viscosity dependent parameters, hence the name. In \cite{Dor99b} a Galerkin method with downwinded elements is analysed. We, on the other hand, are interested in a Petrov-Galerkin method with a standard trial space, and an upwinded test space.  Compared with for instance \cite{DemHeu13} we point out that our methods produce conforming spaces.

One of our main sources of inspiration for our stability proof is \cite{BaiBre83}. This paper analyses a parabolic problem with an $\rmH^{1/2}$-norm in time. For our purposes the time-variable corresponds to a space-variable which increases in the direction of the flow. The Hilbert transform, which provides an inf-sup estimate for the convective term, plays a prominent role in our arguments. Many of our arguments use, in intermediate steps, a slightly weaker variant of $\rmH^{1/2}$, a certain critical Besov space containing discontinuous functions. Both projection onto piecewise constants and interpolation onto upwinded functions behave quite well in this norm, as we show.

We point out that also \cite{San08} advocates the use of a $\rmH^{1/2}$-norm for convection diffusion problems. The theory of that paper pertains to a posteriori estimates for one-dimensional problems, but numerical results are reported also for multi-dimensional problems. In \cite{Can06} an anisotropic $\rmH^{1/2}$-norm is also considered, in a multi-dimensional setting. The numerical methods analysed in \cite{Can06} and \cite{San08} (Fourier/Wavelets and SUPG respectively) are quite different from the one we consider here.

For a certain choice of viscosity dependent norms, essentially the anisotropic $\rmH^{1/2}$-norm plus the energy norm, we prove both a continuity estimate and an inf-sup condition, for the discrete method, up to a logarithmic factor in the viscosity. Admittedly the hypothesis required for our proof, essentially that the flow is aligned with the mesh, is very restrictive. On the other hand, we do prove stability under hypotheses that allow for the formation of a boundary layer at the outflow boundary. It is our hope that the paper might give a reasonable idea of what sort of arguments should be improved upon to handle more realistic meshes.

The paper is organized as follows. In \S \ref{sec:setup} we set up the model problem we consider and discuss some numerical results. In \S \ref{sec:parabolic} we provide a study,  based on \cite{BaiBre83}, of some parabolic problems, to motivate our techniques. In \S \ref{sec:cont} we provide continuity estimates, in the norms of interest, for some operators acting on functions of one real variable, in particular projection onto piecewise constants and the nodal interpolator, acting from piecewise affines to upwinded functions. In \S \ref{sec:convdiff} we put our results together to prove an inf-sup condition for convection diffusion problems. An appendix contains our first proof of this result, which was more cumbersome.

\section{\label{sec:setup} Problem setup}

We consider a domain $U$ in $\bbR^n$. On this domain we consider the equation:
\begin{equation}\label{eq:convdiff}
-\alpha \Delta u + \beta \cdot \nabla u +  \gamma u = f.
\end{equation}
The scalar $\alpha >0$ is constant in the domain and will be referred to as viscosity. The vector field $\beta$, responsible for convection, is also constant in the domain, and directed along the first axis. With a slight abuse of notations we take it of the form $\beta e$ for a scalar $\beta > 0$ and $e$ the unit vector along the fist axis. The function $\gamma$ is bounded and non-negative.   The right hand side $f$ is given in $\rmL^2(U)$. We impose homogeneous Dirichlet conditions on $u$. This equation has a unique solution in $\rmH^1_0(U)$, as can be deduced from the Lax-Milgram lemma. 

We are interested in letting the parameter $\alpha$ tend to $0$, all other data remaining fixed. If we let $u_\alpha$ denote the corresponding solution, we know that $u_\alpha$ converges in $\rmL^2(U)$ to some function $u_0$ as $\alpha$ tends to $0$. It then follows that:
\begin{equation}
\alpha \int |\nabla u_\alpha |^2  + \int \gamma |u_\alpha|^2 \to \int f u_0.
\end{equation}
In general therefore the $\rmH^1(U)$-seminorm of $u_\alpha$ blows up. One observes the formation of a boundary layer at the outflow boundary, which is the part of $\partial U$ where $\beta \cdot \nu > 0$, where $\nu$ denotes the outward pointing normal vector on $\partial U$. A boundary layer of a rather different nature appears close to the part of the boundary where $\beta \cdot \nu = 0$. Away from the boundary layers, $u_\alpha$ converges to $u_0$ in strong norms. The limit $u_0$ satisfies the homogeneous boundary condition on the inflow boundary (where $\beta \cdot \nu <0$), but in general not elsewhere.

\begin{figure}
\includegraphics[width=7cm]{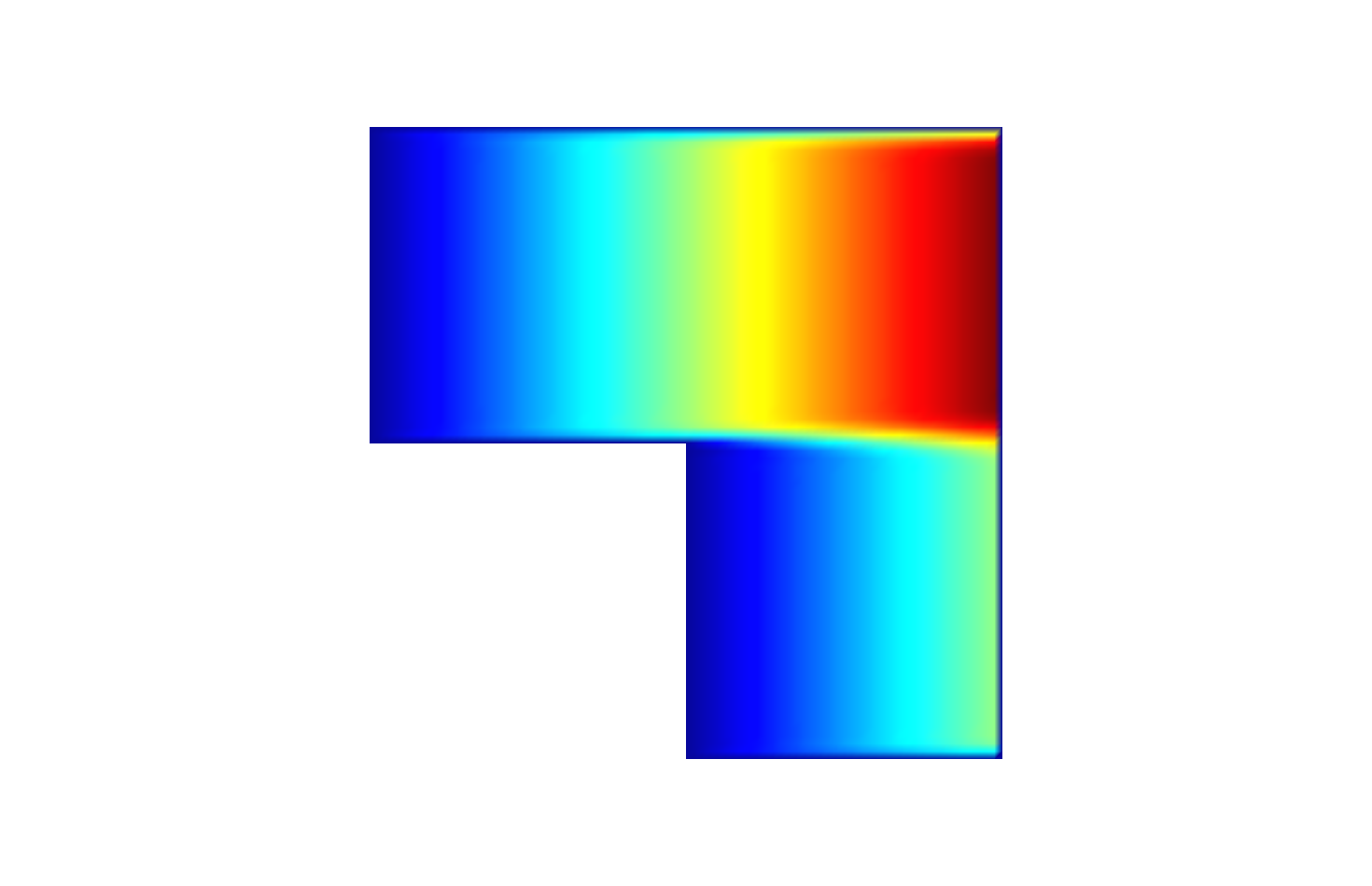}
\includegraphics[width=7cm]{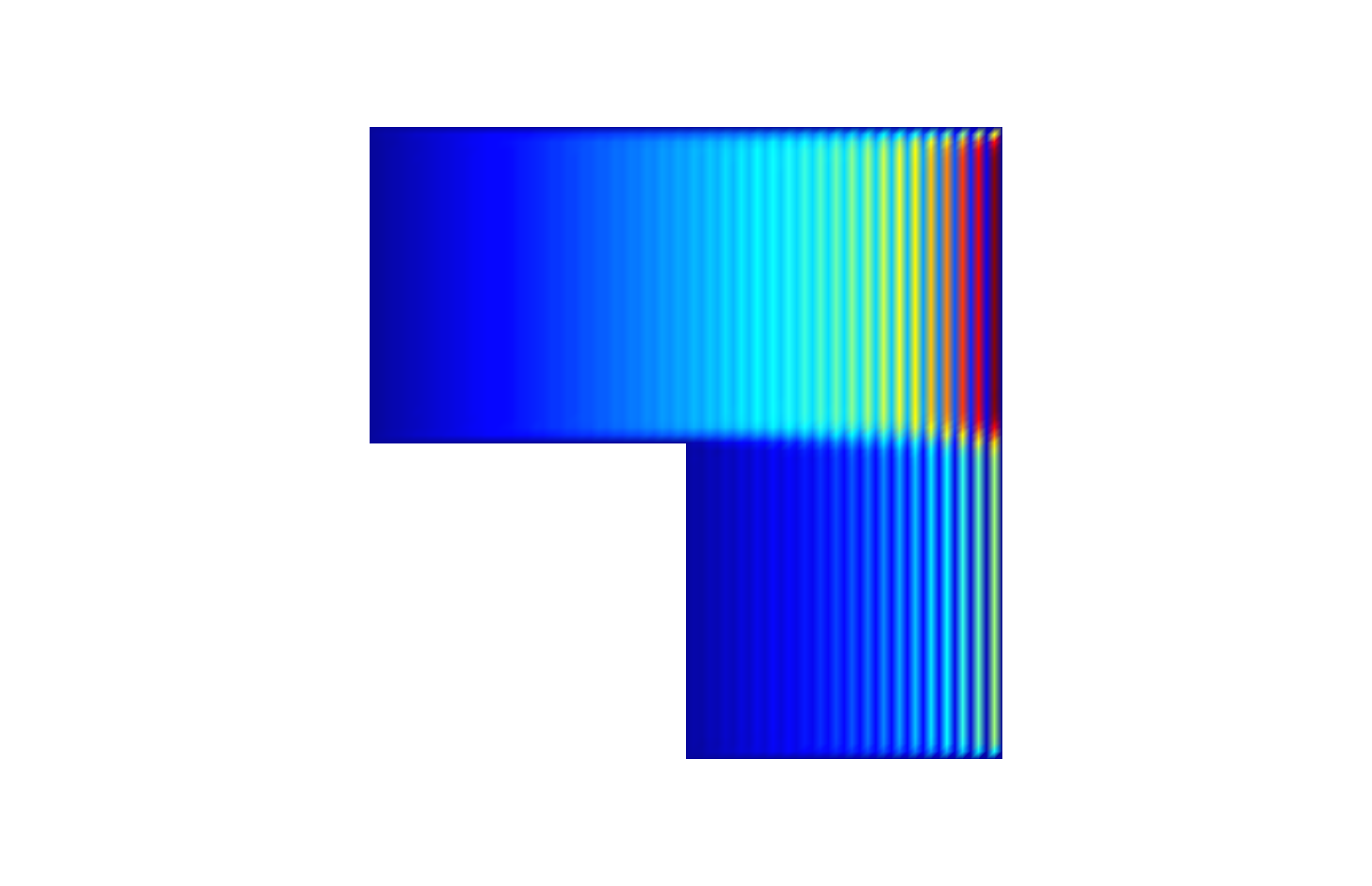}
\caption{\label{fig:wwo} Numerical solution computed with (left) and without (right) our proposed upwinding.}
\end{figure}
In Figure \ref{fig:wwo} numerical results are shown for $f = 1$, $\beta = 1$ and $\gamma = 0$. The width of the domain in the horizontal direction is 1. We chose $\alpha = 3 \times 10^{-4}$. The standard numerical method we use, is a  Galerkin finite element method with continuous $Q_{11}$ finite elements on a square grid of width $\sigma$.  For $\sigma=1/80$ we observe that the numerical solution is very oscillatory (right hand figure). This well-known instability appears whenever the P\'eclet number $\beta \sigma /\alpha$ exceeds 1 (for the displayed figure it is above 40). The upwinded method we propose is a Petrov-Galerkin method with $Q_{11}$ as trial space and a test space we now proceed to describe.

Generally speaking, consider a mesh consisting of cells of various dimensions: vertices (dimension $0$), edges (dimension $1$), faces (dimension $2$), etc. arranged in a cellular complex. For definiteness one can think of simplicial complexes or product grids. Actually only the latter are considered in our numerical experiments and for the stability proof we present in the following sections.
 
For each cell $T$, of any dimension, let $\beta_T$ be the tangential component of $\beta$ on $T$. We construct an upwinded basis function $v$ attached to a given vertex by first assigning the value 1 to this vertex, and the  value 0 to all others. Next we extend recursively, from vertices to edges, from edges to faces, etc, each time solving the equation, on say the cell $T$ (with prescribed boundary values):
\begin{equation}\label{eq:harmext}
-\alpha \Delta_T v - \beta_T \cdot \nabla_T v = 0.
\end{equation}

We may remark that if $\beta = 0$ the method of recursive harmonic extension produces the piecewise affines on simplicial meshes, and the tensor product $Q_{1\ldots 1}$ functions on orthogonal product meshes. On such meshes the obtained functions are simple also in the case where $\beta $ is non-zero (but constant on the domain): On an edge the solutions to (\ref{eq:harmext}) are linear combinations of the constant function and a certain exponential:
\begin{equation}\label{eq:cpe}
v(x)= c_1 +c_2 \exp(-\frac{\beta_T}{\alpha} x).
\end{equation}
If $\beta_T= 0$ one replaces of course the exponential by a linear function. Globally one obtains, from the recursive extension procedure, tensor products of such functions.

We develop our theory for the case when the domain is of the form  $U =  ]0,T[ \times V$ and the mesh is aligned with the first axis, which is also the direction of the vector field $\beta$. In this case we obtain, with the above method, the tensor products of functions which in the direction of the first axis are piecewise of the form (\ref{eq:cpe}), and which in the direction of $V$ are standard $Q_{1...1}$ functions. This defines our upwinded test space. Recall that the trial space is just $Q_{1...1}$.

For applications it is important that the numerical method be able to treat variable $\beta$, not necessarily aligned with the mesh. This is done by replacing (\ref{eq:harmext}) by :
\begin{equation}\label{eq:harmextb}
\div_T \exp(\frac{\beta_T \cdot x}{\alpha}) \grad_T v (x) = 0,
\end{equation}
and solving this equation approximately on a sub-grid. The particular sub-grid we advocate consists in adding one point to each cell of the mesh, and taking the corresponding simplicial refinement. The added points are placed taking into account the expected singular behaviour of the upwinded basis functions. In other words we do a barycentric refinement, where barycenters are computed with weights involving the P\'eclet number $\beta \sigma/\alpha$. The main reason for preferring (\ref{eq:harmextb}) to (\ref{eq:harmext}) is that in this form the \emph{discrete} upwinding method (involving typically an adapted subgrid) extends nicely to differential forms, as explained in \cite{Chr13FoCM}.

In Figure \ref{fig:rel} we have plotted the relative distance between the numerical solutions $u_{ex}$ and $u_{ap}$ obtained with exact  and approximate upwinding respectively, with respect to the norm defined by:
\begin{equation}
\| u\|_{\alpha}^2  = \int |u|^2 + \alpha \int |\nabla u|^2 .
\end{equation}
We first notice that the relative error stays below 0.03 in this experiment, indicating that the use of approximate unwinding does not change too much the computed solution. In the following sections we just analyse the case of exact upwinding.

Interestingly, we also notice that the relative error, between exact and discrete upwinding in the computed solutions, seems to be maximal along a certain line, here given by  $\alpha = 0.15 \sigma$. We also noticed that the relative error between the exact and upwinded basis functions is maximal along such a line, but with a different proportionality constant.

\begin{figure}[htbp]
\begin{center}
\includegraphics[width=12cm]{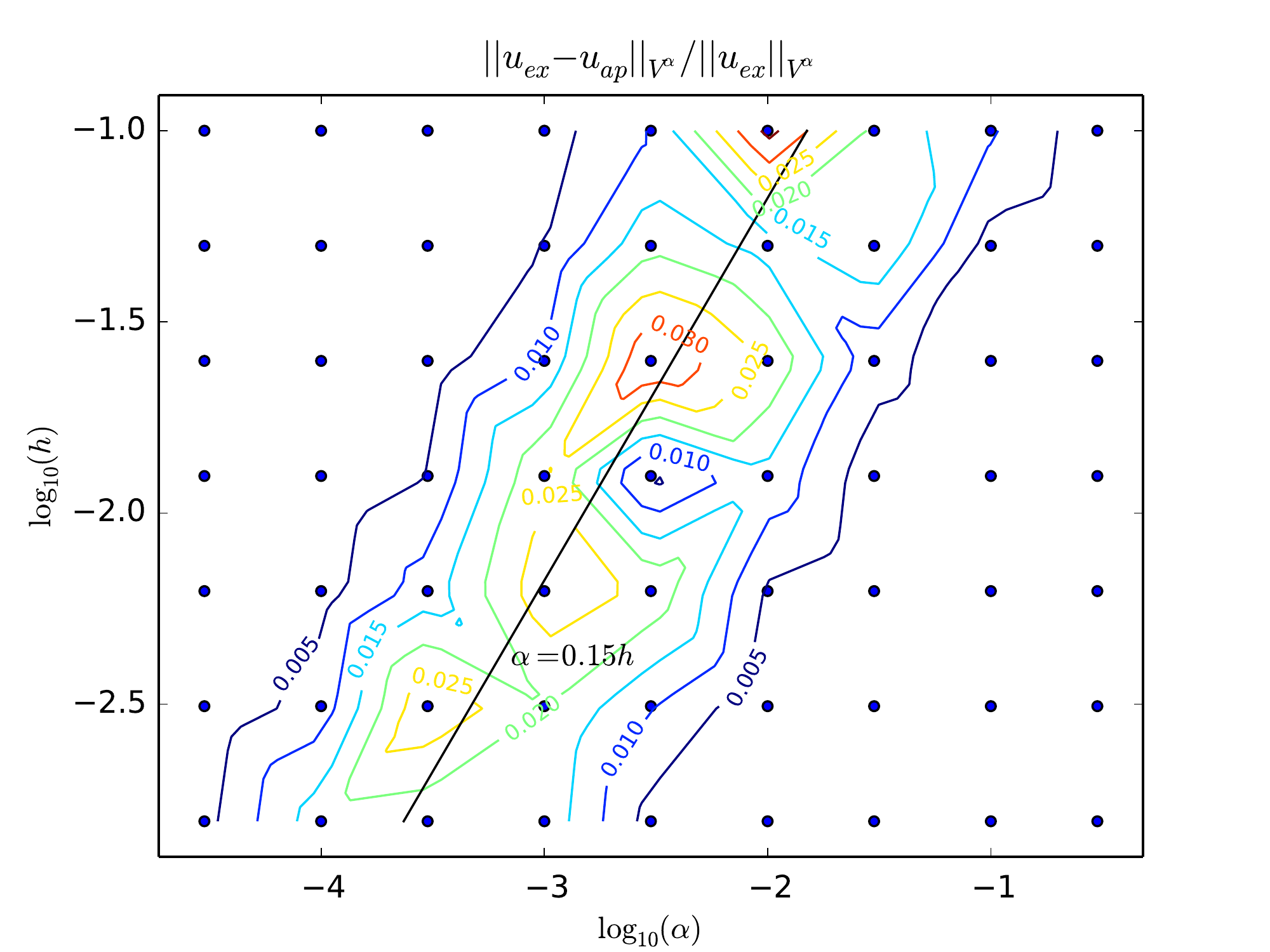}
\end{center}
\caption{\label{fig:rel} Relative distance between the solutions obtained with exact and approximate upwinding}
\end{figure}

\section{\label{sec:parabolic} A study of parabolic problems}

In this section we give our reading of \cite{BaiBre83}. It serves mainly to motivate the techniques of the next sections. Some improvements occur, because in our setting the viscosity is destined to become small, as illustrated by Lemma \ref{lem:equiv}. We shall restrict our attention to the Crank-Nicolson scheme, whereas \cite{BaiBre83} treats some other discretizations as well.

\paragraph{Inf-Sup condition.} We state some ways in which inf-sup conditions may be obtained, for bilinear forms on Hilbert spaces.

\begin{proposition}\label{prop:isa}
Suppose $X$ and $Y$ are Hilbert spaces and that $a:X \times Y \to \bbR$ is a continuous bilinear form.
Suppose we have a continuous linear map $A: X \to Y$ such that:
\begin{align}
\| A u \| &\leq C_1 \| u\|,\\
|a(u, Au)| &\geq \frac{1}{C_2} \|u \|^2.
\end{align}
Then $a$ satisfies the inf-sup condition:
\begin{equation}
\inf_{u \in X} \sup_{v \in Y} \frac{|a(u,v)|}{\| u \| \, \|v\| } \geq \frac{1}{C_1 C_2}.
\end{equation}
\end{proposition}
\begin{proof}
For non-zero $u \in X$ we have $Au \neq 0$ and we may write:
\begin{align}
\sup_{v \in Y} \frac{|a(u,v)|}{\| v\|} \geq \frac{|a(u, Au)|}{\| A u \|} \geq \frac{1}{C_1 C_2} \|u\|.
\end{align}
The inf-sup estimate follows.
\end{proof}

\begin{proposition}\label{prop:combine}
Suppose we have two Hilbert spaces $X$ and $Y$, and two continuous bilinear forms $a$ and $b$ on $X \times Y$. Suppose we have two continuous operators $A$ and $B$ from $X$ to $Y$ such that for some $C_1 > 0$:
\begin{align}
b(u, B u) + a(u, A u) & \geq \frac{1}{C_1} \| u\|^2.
\end{align}
Suppose moreover that we have the compatibility conditions:
\begin{align}
b(u, A u) & \geq 0, \\
|a(u, B u)| & \leq C_2 a(u, Au).
\end{align}
Then $b+a$ satisfies an inf-sup condition on $X \times Y$.
\end{proposition}

\begin{proof}

We introduce a parameter $\lambda >0$. We remark that $B + \lambda A: X \to Y$ is continuous and that:
\begin{align}
(b + a)(u, Bu + \lambda A u)  & = b(u, B u) + \lambda b(u, A u) + a(u, B u ) + \lambda a(u, Au ), \\
& \geq b(u, B u) +  \lambda a(u, A u) - |a(u, B u)|,\\
& \geq b(u, B u) + a(u, A u) + (\lambda - C_2 -1)a(u, A u). 
\end{align}
We choose $\lambda \geq C_2 + 1$. Then we apply Proposition \ref{prop:isa}.
\end{proof}

\paragraph{Crank-Nicolson.}

We let $\bbR_+$ denote the set of positive reals. Recall that $\rmH^{1/2}_{00}(\bbR_+)$ denotes the Lions-Magenes space of scalar functions on $\bbR_+$ whose extension by zero to $\bbR$ are in $\rmH^{1/2}(\bbR)$, see chapter 33 in \cite{Tar07}. Such spaces may also be obtained by interpolation techniques and this will play a role in our arguments. We will use the following seminorm on $\rmH^{1/2}(\bbR)$, defined by the Fourier transform, denoted $\calF$:
\begin{equation}
|u|_{\rmH^{1/2}}^2 = \int |\xi| \, |(\calF u)(\xi)|^2 \rmd \xi.
\end{equation}
The full norm on $\rmH^{1/2}(\bbR)$ is obtained by adding the $\rmL^2$-norm.

We let $O$ be a Hilbert space, with scalar product $\langle \cdot, \cdot \rangle$.
Let $X$ be a Hilbert space contained in $O$.
Let $a: X \times X \to \bbR$ be a continuous symmetric bilinear form, which is also coercive.

We also use spaces of Hilbert-space valued functions of a real variable, such as the Bochner space $\rmL^2(\bbR_+, X)$, without  further ado. The real variable will usually be referred to as time. The time-derivative of a function $u$ is denoted $\dot u$.

We define a Hilbert space $Y$ as follows:
\begin{equation}\label{eq:defy}
Y_{00} = \rmH^{1/2}_{00}(\bbR_+, O) \cap \rmL^2(\bbR_+, X).
\end{equation}
As a variant, we also use the space:
\begin{equation}\label{eq:defyvar}
Y = \rmH^{1/2}(\bbR_+, O) \cap \rmL^2(\bbR_+, X).
\end{equation}

Given $f \in Y'$, we are interested in finding $u \in Y_{00}$ such that for all $v$:
\begin{equation}\label{eq:abstractp}
\langle \dot u, v \rangle + a(u,v) = \langle f, v \rangle,
\end{equation}
in some weak sense which may involve integration in time. This is an abstract parabolic equation. The initial condition $u(0)= 0$ is imposed in a weak sense by these hypotheses.

Consider a family of Galerkin spaces $X_\sigma$, included in $X$, where the parameter $\sigma$ is thought of as mesh-width. Given also a time-step $\tau$, the Crank-Nicolson scheme is defined as follows. We let $u: \bbR_+ \to X_\sigma$ be continuous and $\tau$-piecewise affine. We denote $u_i= u(i\tau)$, for $i \in \bbN$. We impose $u(0)= 0$ and, for all $v \in X_\sigma$:
\begin{equation}
\langle \frac{u_{i+1} - u_i}{\tau}, v \rangle + a(u_{i+1/2},v) = \langle f_{i+1/2}, v \rangle.
\end{equation} 
Here we have put:
\begin{equation}\label{eq:uih}
u_{i+1/2}= \frac{1}{2}(u_i + u_{i+1}),
\end{equation}
and:
\begin{equation}\label{eq:fih}
f_{i+1/2}= \frac{1}{\tau}\int_{i \tau}^{(i+1)\tau} f.
\end{equation}
For any function $u$ on $\bbR$, we let $\overline u$ be the function which is $\tau$-piecewise constant, with the same piecewise averages as $u$. In other words $\overline u$ is the $\rmL^2$ projection of $u$ onto the piecewise constants. This generalizes both (\ref{eq:uih}) and (\ref{eq:fih}).

The Crank-Nicolson scheme then yields, for all functions $v: \bbR_+ \to X_\sigma$:
\begin{equation}\label{eq:cncont}
\int \langle \dot u, v \rangle + \int a(\overline u , v) = \int \langle \overline f, v \rangle.
\end{equation}
Notice in particular that, even though $u$ is piecewise affine in time, this identity holds for any time-dependence of $v$ (remaining, say, integrable).

\paragraph{Small abstract viscosity.} We are particularly interested in problems with a parameter $\alpha\in \bbR_+$. That is, the space $X$ is replaced by $X^\alpha$, and $a$ by $a^\alpha$. As a set $X = X^\alpha$, but we suppose that the norm of $X^\alpha$ is equivalent to the one defined by $a^\alpha$, uniformly in $\alpha$. Explicitly, there is a constant $C >0$ such that for all $\alpha$ and all $w \in X$:
\begin{equation}\label{eq:alphaequiv}
\frac{1}{C} \| w\|_\alpha^2 \leq a^\alpha(w,w) \leq C \| w \|_\alpha^2.
\end{equation}
 As $\alpha$ tends to $0$, the norm $\| \cdot \|_\alpha$ converges to the norm on $O$, which is henceforth denoted $\| \cdot \|_0$. We denote by $Y^\alpha_{00}$ and $Y^\alpha$ the corresponding modifications of $Y_{00}$ and $Y$.

In accordance with the interpretation of $\sigma$ as mesh-width for a finite element method, and $\alpha$ as a viscosity parameter, we suppose that the following inverse inequality holds. There exists $C>0$ such that, for all $\alpha$ and $\sigma$, and all $u\in X_\sigma$: 
\begin{equation}
\|u\|_{\alpha} \leq C (1 + \alpha^{1/2}\sigma^{-1}) \| u\|_0.
\end{equation}
For notational convenience, we denote statements of the form $A \leq CB$ for some large enough $C$  independent of the parameters in their natural range, as $A\cleq B$.

\begin{lemma}\label{lem:equiv}
Suppose we have estimates $\tau \cleq \sigma$ and $\alpha \cleq \sigma$. Then we have an estimate, for $u$ $\tau$-piecewise affine with values in $X_\sigma$:
\begin{equation}
| u|_{\rmH^{1/2}(O)}^2 + \| u\|_{\rmL^2(X^\alpha)}^2 \cleq  |u|_{\rmH^{1/2}(O)}^2 + \| \overline u\|_{\rmL^2(X^\alpha)}^2.
\end{equation}
\end{lemma}
\begin{proof}
We have, using first an approximation estimate and then the above inverse inequality:
\begin{align}
\| u - \overline u\|_{\rmL^2(X^\alpha)} & \cleq \tau^{1/2} |u|_{\rmH^{1/2}(X^\alpha)},\\
&\cleq \tau^{1/2}(1+\alpha^{1/2}\sigma^{-1})|u|_{\rmH^{1/2}(O)},\\
&\cleq |u|_{\rmH^{1/2}(O)}.
\end{align}
From this, the announced estimate follows.
\end{proof}

\paragraph{Stability of Crank-Nicolson.} We now derive a stability estimate for the Crank-Nicolson scheme for parabolic problems, reformulated by (\ref{eq:cncont}). 

For this purpose we will use the Hilbert transform, which is convolution by the kernel function $x \mapsto 1/x$. Since the kernel function is non-integrable, the definition of the Hilbert transform on test functions involves taking a principal value. The Fourier transform of the kernel function is some constant times the sign function. Multiplying the Hilbert transform by a suitable constant, we obtain an operator $\calH$ satisfying the following crucial identity:
\begin{equation}
\int \dot u \calH u = |u|^2_{\rmH^{1/2}}. 
\end{equation}
Various continuity properties of the Hilbert transform will also be used, for which we refer to \cite{Tar07}. In particular it is continuous inside Hilbert spaces $\rmH^s(\bbR)$ for all $s$. Henceforth we take the liberty of calling $\calH$ the Hilbert transform. 

Following \cite{BaiBre83}, we intend to apply a variant of Proposition \ref{prop:combine}, but we consider the regime $\tau \cleq \sigma$ and $\alpha \cleq \sigma$, where we may use Lemma \ref{lem:equiv}. Recall that the proposition features two Hilbert spaces $X$ and $Y$ and two operators $A$ and $B$ from $X$ to $Y$. The space $X$ will be the one defined by $(\ref{eq:defy})$, whereas $Y$ will be defined by  (\ref{eq:defyvar}). The operator $A$ will be the identity and $B$ the Hilbert transform (composed with restriction to the half-line). The bilinear forms $b$ and $a$ in that proposition will correspond to the first term and second term on the left hand side of (\ref{eq:abstractp}), integrated in time. 

Considering now an element $u$ of the space defined in $(\ref{eq:defy})$ we define a candidate for optimal test function $v = \calH u + \lambda u$, for some large enough $\lambda >0$ to be determined.. We suppose that $u$ is piecewise affine with values in $X_\sigma$. Then $v$ also takes values in $X_\sigma$. For the bilinear form appearing on the left hand side of $(\ref{eq:cncont})$ we get:
\begin{align}
\int \langle \dot u, v \rangle + \int a^\alpha(\overline u , v) \geq
| u |_{\rmH^{1/2}(O)}^2 + \lambda \int a^\alpha( \overline u , \overline u) - \int |a^\alpha(\overline u , \calH u)|. 
\end{align}
We have a continuity estimate, for any $\epsilon >0$:
\begin{align}
\int |a^\alpha(\overline u , \calH u)|  & \cleq  \|\overline u\|_{\rmL^2(X^\alpha)} \| \calH u\|_{\rmL^2(X^\alpha)},\\
& \leq C \|\overline u\|_{\rmL^2(X^\alpha)} \| u\|_{\rmL^2(X^\alpha)},\\
& \leq  \frac{C}{2 \epsilon}  \|\overline u\|_{\rmL^2(X^\alpha)}^2 + \frac{ C\epsilon }{2 }\| u\|_{\rmL^2(X^\alpha)}^2.
\end{align}
So we get, using Lemma \ref{lem:equiv}:
\begin{align}
\int \langle \dot u, v \rangle + \int a^\alpha(\overline u , v)  \geq &\frac{1}{C'} \left(| u|_{\rmH^{1/2}(O)}^2 + \| u\|_{\rmL^2(X^\alpha)}^2\right) +\\
& (\lambda -1) \int a^\alpha( \overline u , \overline u) -  \frac{C }{2 \epsilon}  \|\overline u\|_{\rmL^2(X^\alpha)}^2 - \frac{ C\epsilon }{2 }\| u\|_{\rmL^2(X^\alpha)}^2. \label{eq:seclin}
\end{align}
To control the last term, choose $\epsilon$ so small that :
\begin{equation}
\frac{ C\epsilon }{2 } < \frac {1}{C'}.
\end{equation}
Then, to handle the two other terms on line (\ref{eq:seclin}), choose $\lambda$ so big that:
\begin{equation}
(\lambda -1)  > \frac{C}{2 C'' \epsilon},
\end{equation}
where $C''$ is the constant appearing in $(\ref{eq:alphaequiv})$.

For such a choice of $\epsilon$ and $\lambda$, we get an estimate:
\begin{equation}
\int \langle \dot u, v \rangle + \int a^\alpha(\overline u , v) \cgeq  | u|_{\rmH^{1/2}(O)}^2 + \| u\|_{\rmL^2(X^\alpha)}^2.
\end{equation}
Together with (\ref{eq:cncont}) this gives a stability estimate for the Crank-Nicolson scheme:
\begin{equation}
\| u\|_{Y_{00}^\alpha}^2 = |u|_{\rmH^{1/2}(O)}^2 + \| u\|_{\rmL^2(X^\alpha)}^2 \cleq \| \overline f\|^2_{(Y^{\alpha})'}.
\end{equation}
Notice the somewhat annoying fact that on the right hand side, we have $\overline f$ where we would have preferred $f$. The operation $f \to \overline f$ is well defined on the (non-closed) subspace $\rmL^2(\bbR_+, (X^{\alpha})')$ of $(Y^{\alpha})'$. However it is not well defined inside $\rmH^{1/2}(\bbR_+)$, nor its dual.

\begin{remark}
In \cite{BaiBre83}, which treats the case of fixed $\alpha$, the obtained stability estimate concerns the discrete ($\tau$-dependent) norm with square:
\begin{equation}
|u|_{\rmH^{1/2}(O)}^2 + \| \overline u\|_{\rmL^2(X)}^2.
\end{equation}
The appearance of $\overline u$ makes this norm slightly weaker. The interpretation is that some oscillations in time are not so well controlled in $X$-norm, even though the first term controls them in $O$-norm.
\end{remark}

\section{\label{sec:cont} Stability estimates of some operators}

In this section we prove uniform continuity estimates for the operators we use in our stability proofs. These operators act on functions of a real variable. We will use several variants of the space $\rmH^{1/2}(\bbR)$, where the functions take values in Hilbert spaces. We refer to \cite{Tar07} (especially chapter 35) for definitions pertaining to scalar-valued functions. Until now we have used the characterisation with the Fourier transform, but now we will also use the Slobodetski seminorm, as it appears in particular in Lemma 35.2 in \cite{Tar07}. Specifically, on an interval $I$, for $s\in ]0,1[$:
\begin{equation}
| u |_{\rmH^{s}(I)}^2  = \iint_{I \times I} \frac{| u(x + y) - u(x)|^2}{| y|^{1 + 2 s}} \rmd x \rmd y.
\end{equation}
For $I = \bbR$ and for $s$ in a compact subset of $]0,1[$, this seminorm is uniformly equivalent to the one defined by Fourier transform.

We will use also the space $\rmH^{1/2}_w(\bbR)$ consisting of functions $u \in  \rmL^2(\bbR)$ such that for some $C \geq 0$, we have, for all $y \in \bbR$:
\begin{equation} 
\| u - \tau_y u\|_{\rmL^2} \leq C |y|^{1/2}.
\end{equation}
Here, $\tau_y$ denotes translation by the vector $y$. The best constant $C$ in this estimate defines a seminorm, denoted:
\begin{equation}
| u|_{\rmH^{1/2}_w} = C.
\end{equation}
The subscript $w$ stands for "weak", reflecting that the Banach space $\rmH^{1/2}_w(\bbR)$ is slightly bigger than the Hilbert space $\rmH^{1/2}(\bbR)$. The space $\rmH^{1/2}_w(\bbR)$ is nothing but the Besov space $B^{1/2, 2}_\infty(\bbR)$, see for instance Lemma 35.1 in \cite{Tar07}. We will use that it is big enough to contain piecewise constant functions. On the other hand it is small enough to be included in all the spaces $\rmH^{1/2-\epsilon}(\bbR)$ for $\epsilon >0$.

We will use two methods to obtain upwinded functions:
\begin{itemize}
\item First $\rmL^2$-project onto piecewise constants and then convolve with $G_\alpha$, which will be defined later. We notice that values at extremities will not be zero in general, so this needs to be taken care of.
\item Interpolate the values at vertices.
\end{itemize}
We now analyse the stability of these two methods, one after the other.

\paragraph{Projection onto piecewise constants and convolution.}

\begin{proposition}\label{prop:bar}
Consider the map $u \mapsto \overline u$, which $\rmL^2$-projects onto $\tau$-piecewise constants. It is bounded from $\rmH^{1/2}(\bbR)$ to $\rmH^{1/2}_w(\bbR)$, uniformly in $\tau$.
\end{proposition}
\begin{proof}
In this proof we use the Slobodetski seminorm on $\rmH^{1/2}(I) = \rmH^{1/2}(a,b) $, for various intervals $I= ] a, b[$. We let $P_\tau: u \mapsto \overline u$ denote the $\rmL^2$ projection onto $\tau$-piecewise constant functions.

On the reference interval $]-1, 1 [$ we have an estimate, for the jump at $0$:
\begin{equation}
|(P_1 u) (0+) - (P_1 u)(0-)| \cleq |u|_{ \rmH^{1/2}(-1,1)}.
\end{equation}
Scaling to the interval $]-\tau, \tau [$ one notices that the two sides scale in the same way, so we get:
\begin{equation}
|(P_\tau u) (0+) - (P_\tau u)(0-)| \cleq |u|_{ \rmH^{1/2}(-\tau, \tau)}.
\end{equation}
For $|y| < \tau$ we have:
\begin{align}
\| P_\tau u - \tau_y P_\tau u \|_{\rmL^2}^2  & \cleq \sum_{k \in \bbZ} |y| \, |(P_\tau u)(k\tau+) - (P_\tau u )(k \tau-)|^2,\\
& \cleq |y| \sum_{k \in \bbZ} |u|_{\rmH^{1/2}((k-1)\tau, (k+1)\tau )}^2,\\
& \cleq |y| \, |u|_{\rmH^{1/2}(\bbR)}^2.
\end{align}
For $|y| \cgeq \tau$ we have:
\begin{align}
\| P_\tau u - \tau_y P_\tau u \|_{\rmL^2} & \leq \| P_\tau u -u \|_{\rmL^2} + \|u -\tau_y u \|_{\rmL^2} + \| \tau_y u - \tau_y P_\tau u \|_{\rmL^2},\\
& \cleq (\tau^{1/2} + |y|^{1/2} ) |u|_{\rmH^{1/2}},\\
& \cleq |y|^{1/2}  |u|_{\rmH^{1/2}}.
\end{align}
Together these two estimates conclude the proof.
\end{proof}

\begin{proposition}\label{prop:wh}
Consider the canonical injection of $\rmH^{1/2}_w$ into $\rmH^{1/2 - \epsilon}$. Its norm is of order $1/ \epsilon^{1/2}$ for small $\epsilon$. 
\end{proposition}
\begin{proof}
We write:
\begin{align}
| u |_{\rmH^{1/2 - \epsilon}}^2  & = \iint \frac{| u(x + y) - u(x)|^2}{| y|^{ 2 - 2 \epsilon}} \rmd x \rmd y, \\
& \cleq \left(\int \frac{\min\{ |y|, 1 \}}{| y|^{ 2 - 2 \epsilon}} \rmd y\right) \| u\|_{\rmH^{1/2}_w}^2.
\end{align}
The integral over $y$ is bounded by:
\begin{equation}
\int_0^1 \frac{1}{[y|^{1- 2 \epsilon}}  \rmd y+ \int_1^\infty \frac{1}{[y|^{2- 2 \epsilon}} \rmd y= \frac{1}{2 \epsilon} + \frac{1}{1 - 2 \epsilon}.
\end{equation}
This yields the claimed result.
\end{proof}

Given a locally integrable function $u$ on $\bbR$, such as a piecewise constant one, we are interested in finding an absolutely continuous function $v$ that solves:
\begin{equation}
\alpha \dot v + \beta v = \beta u.
\end{equation}
We may determine $v$ as:
\begin{equation}
v(t) = \int_{- \infty}^t \frac{\beta}{\alpha} \exp( \frac{\beta(s-t)}{\alpha}) u (s) \rmd s.
\end{equation}
We introduce the function $G_\alpha$ defined by:
\begin{equation}
G_\alpha(s) = \left \{ \begin{array}{l l}
\frac{\beta}{\alpha} \exp( \frac{- \beta s}{\alpha}) &\textrm{ for } s \geq 0,\\
0 &\textrm{ for } s < 0.
\end{array} \right.
\end{equation}
With this notation we have:
\begin{equation}
v = G_\alpha \ast u.
\end{equation}
We now provide some mapping properties of convolution by $G_\alpha$. We notice first that:
\begin{equation}
\| G_\alpha \|_{\rmL^1} = 1.
\end{equation}
This gives uniform boundedness, from $\rmL^2(\bbR)$ to $\rmL^2(\bbR)$, for convolution by $G_\alpha$. Notice that this estimate works also for Hilbert space valued functions.
\begin{proposition}\label{prop:diffnorm}
The map $ u \mapsto G_\alpha \ast u$,  from $\rmH^{1/2 - \epsilon}(\bbR)$ to $\rmH^{1/2}(\bbR)$, has a norm of order
$1/\alpha^{\epsilon}$ for small $\epsilon$, uniformly in $\alpha$.
\end{proposition}
\begin{proof}
For this proof we suppose, without loss of generality, that $\beta = 1$.

\noindent The Fourier transform of $G_\alpha$ is given by:
\begin{equation}
\calF G_\alpha (\xi) = \frac{1}{1 + \alpha i \xi}.
\end{equation}
It follows that:
\begin{equation}
|\calF G_\alpha (\xi)|^2 = \frac{1}{1 + \alpha^2 |\xi|^2}.
\end{equation}
We can therefore write:
\begin{align}
|v|_{\rmH^{1/2}}^2 & = \int \frac{|\xi|}{1 + \alpha^2|\xi|^2} |\calF u (\xi)|^2 \rmd \xi, \\
& \leq C(\alpha, \epsilon) |u|_{\rmH^{1/2-\epsilon}}^2,
\end{align}
with:
\begin{equation} 
C(\alpha,\epsilon) = \max \{ \frac{ |\xi|^{2 \epsilon}}{1 + \alpha^2|\xi|^2} \ : \ \xi \in \bbR \}.
\end{equation}
Calculus gives that the maximum is achieved when:
\begin{equation}
|\xi|^2 = \frac{\epsilon}{(1- \epsilon) \alpha^2 }.
\end{equation}
This provides:
\begin{equation}
C(\alpha, \epsilon) = \frac{(1-\epsilon) \epsilon^\epsilon}{(1- \epsilon)^{\epsilon} }\frac{1}{\alpha^{2\epsilon}}.
\end{equation}
One checks:
\begin{equation}
\lim_{\epsilon \to 0} \frac{(1-\epsilon) \epsilon^\epsilon}{(1- \epsilon)^{\epsilon} } = 1.
\end{equation}
This concludes the proof.
\end{proof}

\begin{corollary}\label{cor:comp} When we compose the three operators defined in Propositions \ref{prop:bar}, \ref{prop:wh} and \ref{prop:diffnorm}, which consists in projecting onto $\tau$-piecewise constants and then convolving with $G_\alpha$, we get an operator from $\rmH^{1/2}(\bbR)$ to itself, with norm of order $|\log(\alpha)|^{1/2}$.
\end{corollary}
\begin{proof} We get a norm of order $1/ (\epsilon^{1/2} \alpha^{\epsilon})$ and choose $\epsilon =  1 / |\log (\alpha)|$.
\end{proof}

We also require the following result, whose proof is a variant of the above techniques:
\begin{proposition}\label{prop:varal} Given $u\in \rmH^{1/2}(\bbR)$ let $\overline u$ be the projection onto $\tau$-piecewise constants and define $v= G_\alpha \ast \overline u$. Then we have  bound:
\begin{equation}
\alpha \int |\dot v|^2 \cleq |\log(\alpha)| \|u\|_{\rmH^{1/2}}^2.
\end{equation}
\end{proposition}
\begin{proof}
We write:
\begin{align}
\alpha \int |\dot v|^2 & \cleq \alpha \int |\xi|^2 |(\calF G_\alpha \ast \overline u)(\xi)|^2 \rmd \xi,\\
& \cleq \int \frac{\alpha |\xi|^2 }{1 + \alpha^2|\xi|^2} |(\calF \overline u) (\xi)|^2 \rmd \xi,\\
& \cleq C(\alpha, \epsilon) \| \overline u\|_{\rmH^{1/2-\epsilon}}.
\end{align}
Here we estimate:
\begin{equation}
C(\alpha, \epsilon) = \max \{ \frac{\alpha |\xi|^{1+ 2\epsilon } }{1 + \alpha^2|\xi|^2}   \ : \ \xi \in \bbR \} \cleq \frac{1}{\alpha^{2\epsilon}}.
\end{equation}
We proceed using Propositions \ref{prop:wh} and \ref{prop:bar}.
\begin{equation}
\alpha \int |\dot v|^2 \cleq \frac{1}{\epsilon \alpha^{2\epsilon}} \| u\|_{\rmH^{1/2}}^2. 
\end{equation}
And finally we choose $\epsilon = 1/|\log(\alpha)|$.
\end{proof}

Recall that convolution by $G_\alpha$ on piecewise constants produces upwinded functions that do not respect homogeneous Dirichlet boundary conditions. To control the boundary values we will use:

\begin{proposition}\label{prop:linfal}
 We have an estimate, valid for $u\in \rmH^1(\bbR)$:
\begin{equation}
\| u\|_{\rmL^{\infty}} \cleq |\log(\alpha)|^{1/2} ( \| u\|_{\rmH^{1/2}}^2 + \alpha \int |\dot u|^2)^{1/2}.
\end{equation}
\end{proposition}
\begin{proof}
We write:
\begin{align}
\| u\|_{\rmL^{\infty}} & \cleq \| \calF u \|_{\rmL^1} = \int \frac{(1 + |\xi| + \alpha |\xi|^2)^{1/2}}{ (1 + |\xi| + \alpha |\xi|^2)^{1/2} }|\calF u (\xi)| \rmd \xi.
\end{align}
Here we have prepared for a Cauchy-Schwartz inequality. We are led to evaluate the integral:
\begin{align}
 \int \frac{1}{1 + |\xi| + \alpha |\xi|^2 } \rmd \xi . 
\end{align}
We distinguish two subdomains for the variable $\xi$ by comparing $\alpha |\xi|$ with $1$.
For the first integral we use:
\begin{equation}
 \int_0^{1/\alpha} \frac {1}{1 +  \xi}\rmd \xi \leq 1 + | \log(\alpha) |.
\end{equation}
For the second integral we use:
\begin{equation}
 \int_{1/\alpha}^\infty \frac {1}{ \alpha |\xi|^2} \rmd \xi = 1.
\end{equation}
This concludes the proof.
\end{proof}

\paragraph{Interpolation onto upwinded functions.}

\begin{proposition} For $u$ continuous piecewise affine, $0$ at extremities, and $v$ the upwinded interpolant, we have (independent of $\alpha$ and $\tau$):
\begin{equation}
\int |v|^2  \ceq \int |u|^2.\\
\end{equation}
\end{proposition}
Notice that one of the two bounds is false if we remove the boundary conditions.

\begin{proposition} For $u\in \rmH^1(0,T)$ and $v$ the upwinded interpolant, we have (independent of $\alpha$ and $\tau$):
\begin{equation}
\int | u - v|^2  \cleq \tau^2 \int |\dot u|^2.
\end{equation}
\end{proposition}
\begin{proof}
Because the values of $v$ on a $\tau$-interval lie between the values at the extremities (which is a maximum principle for upwinded functions).
\end{proof}

\begin{proposition}\label{prop:phi}
 Let $u$ be continuous piecewise affine and $v$ be the upwinded interpolant. Then:
\begin{equation}
\int |\dot v |^2 = \Phi(p) \int |\dot u|^2,
\end{equation}
where $p=\beta \tau / \alpha $ is the P\'eclet number and:
\begin{equation}
\Phi(p) = \left ( \frac{\exp(p) +1}{\exp(p) -1} \right ) \frac{p}{2}.
\end{equation}
\end{proposition}
\begin{proof}
It suffices to check the identity on a $\tau$-interval. And there it suffices to check it in the case where $v$ is the function defined by:
\begin{equation}
v(t) = \exp(-\frac{\beta t}{\alpha}).
\end{equation}
Then its an elementary computation.
\end{proof}
Notice in particular that $\Phi(p)$ tends to $1$ as $p$ tends $0$ (as expected) and behaves like $p$ as $p$ tends to infinity.

\begin{proposition}  Let $u$ be continuous piecewise affine and $v$ be the upwinded interpolant. Then:
\begin{equation}
| v|_{\rmH^{1/2}_w} \cleq |u|_{\rmH^{1/2}}.
\end{equation} 
\end{proposition}
\begin{proof}
We first write, on the reference interval $]0, 1[$, included in the reference macro-interval $]-1, 2[$, for $u$ which is continuous and affine on the three subintervals, and $v$ its upwinded interpolant, an estimate which is independent of P\'eclet number, for $|y| \leq 1$:
\begin{align}
\| v - \tau_y v \|_{\rmL^2(0,1)} & \cleq |y|^{1/2} \, \int_{-1}^2 |\dot v| ,\\
& \cleq |y|^{1/2} \, |\max u - \min u| ,\\
& \cleq |y|^{1/2} \, | u|_{\rmH^{1/2}(-1, 2)} .
\end{align}
From there one proceeds as in the proof of Proposition \ref{prop:bar}.
\end{proof}

\begin{proposition} We have an estimate, for functions $u \in \rmH^1(\bbR)$, valid for all $\alpha$:
\begin{equation}
|u|_{\rmH^{1/2}}^2 \cleq \| u\|_{\rmL^2}^2 +  |\log(\alpha)|\,  |u|_{\rmH_w^{1/2}}^2 + \alpha \int |\dot u|^2.
\end{equation}
\end{proposition} 
\begin{proof}
We write the Slobodetski seminorm (letting $(-)$ stand for the similar term with $y <0$):
\begin{equation}
|u|_{\rmH^{1/2}}^2 = \left(\int_0^\alpha + \int_\alpha^1 +  \int_1^\infty\right)\left (\int |u(x + y) - u(x)|^2\rmd x \right )\frac{\rmd y}{y^2} + (-).
\end{equation}
On $]0, \alpha[$ we use:
\begin{equation}
\| u - \tau_y u \|_{\rmL^2}^2 \cleq |y|^2 \| \dot u\|_{\rmL^2}^2.
\end{equation}
On $]\alpha, 1[$ we use:
\begin{equation}
\int_\alpha^1 \frac{1}{y} \rmd y = |\log (\alpha)|.
\end{equation}
On  $]1, \infty [$ we use:
\begin{equation}
\int_1^\infty \frac{1}{y^2} \rmd y = 1.
\end{equation}

This gives the three announced terms in reverse order.
\end{proof}

\begin{proposition}\label{prop:asav}
  Let $u$ be continuous piecewise affine and $v$ be the upwinded interpolant. Then:
\begin{equation}
\alpha \int | \dot v|^2 \cleq |u|_{\rmH^{1/2}}^2 +  \alpha \int | \dot u |^2.
\end{equation} 
\end{proposition}
\begin{proof}
We distinguish two regimes, according to the P\'eclet number $p = \beta \tau /\alpha$:

\noindent -- for $p \leq 1$ we have $\Phi(p) \cleq 1$ so :
\begin{equation}
\alpha \int | \dot v|^2 \cleq  \alpha \int | \dot u |^2.
\end{equation}

\noindent -- for $p \geq 1$ we have $\Phi(p) \cleq p $ so :
 \begin{align}
\alpha \int | \dot v|^2 & \cleq  \tau \int | \dot u |^2,\\
& \cleq |u |^2_{\rmH^{1/2}},
\end{align}
from an inverse inequality. 

Together these two regimes give the announced bound.
\end{proof}

Combining the previous propositions we get:
\begin{proposition}\label{prop:halfint}  Let $u$ be continuous piecewise affine and $v$ be the upwinded interpolant. Then:
\begin{equation}
| v |_{\rmH^{1/2}} \cleq \| u\|_{\rmL^2} +   |\log(\alpha)|^{1/2} |u|_{\rmH^{1/2}} +  (\alpha \int |\dot u|^2)^{1/2}.
\end{equation} 
\end{proposition}

For the above propositions, minimal changes occur when we replace functions from $\bbR$ to $\bbR$, by functions from $\bbR$ to some fixed Hilbert space. From there, the techniques can be extended to more complicated situations, such as spaces of the form (\ref{eq:defy}), where values in two different Hilbert spaces are considered.

\section{\label{sec:convdiff} Convection diffusion}

For a function $u$ defined on a domain $U=]0,T[ \times V$, derivation along the first axis will be denoted $u \mapsto \dot u$, and derivation along the remaining axes (in $V$) will be denoted $u \mapsto  \partial u$. For convenience we refer to the first variable as time and the second as space. We are interested in an elliptic boundary value problem on space-time, namely: 
\begin{equation}
-\alpha \Delta u + \beta \cdot \nabla u + \gamma u= f,
\end{equation}
with homogeneous Dirichlet boundary condition on $\partial U$. The convection takes place along the first axis (time). We rewrite the equation as:
\begin{equation}
-\alpha \ddot u + \beta \dot u + \gamma u - \alpha \partial^2 u = f.
\end{equation}
We consider the following setup:
\begin{itemize}
\item $\alpha >0$ is a parameter, varying in the interval $]0, \alpha_0[$, for some $\alpha_0 >0$, keeping in mind that it is the asymptotic behavior as $\alpha \to 0$ that is of biggest interest.
\item $\beta >0$ is a fixed constant.
\item $\gamma$ is a fixed function in $\rmL^\infty(U)$, such that $\gamma (x)\geq 0$ for all $x\in U$.
\end{itemize}

Given a time-step $\tau$, let $Z^0_\tau$ denote the space of $\tau$-piecewise constant functions on $[0,T]$. Also, let $Z^1_\tau$ denote the space of continuous $\tau$-piecewise affine ones, which are $0$ at the extremities of the interval. Finally let $Z^1_\tau(\alpha)$ denote the space of continuous functions which are $\tau$-piecewise of the upwinded form (\ref{eq:cpe}), which are also $0$ at the extremities.

We consider the following variational formulation. Find $u \in Z^1_\tau \otimes X_\sigma$ such that for all $v \in Z^1_\tau(\alpha)\otimes X_\sigma$:
\begin{equation}\label{eq:varfin}
\int \dot u (\alpha \dot v + \beta v) + \int \gamma u v + \alpha \int \partial u \cdot \partial v =  \int f v.
\end{equation}
We set out to prove an inf-sup condition.

\emph{In the following, one should keep in mind the one-dimensional problem, in which $X_\sigma$ is replaced by $\bbR$ and the last term on the left hand side in (\ref{eq:varfin}) disappears. Equivalently one can consider the multi-dimensional problem without crosswind diffusion.  In the transverse direction one then only needs to consider $\rmL^2(V)$ norms.}

\point Given a trial function $u \in Z^1_\tau \otimes X_\sigma$ we define an upwinded test function in three steps.
\begin{itemize}
\item $w = \calH u$ is Hilbert transform along the first axis. Notice by the way that $w \in \rmL^2(\bbR)\otimes X_\sigma$.
\item $\overline w \in Z^0_\tau \otimes X_\sigma$  is the $\rmL^2$-projection of $w$ onto $\tau$-piecewise constant functions on $\bbR$, with values in $X_\sigma$.
\item $v$ solves $\alpha \dot v + \beta v = \beta \overline w + c $, with $v(t,x)= 0$ for $t=0$ or $t=T$  and $x \in V$ and some function $c\in X_\sigma$ considered on $U$ as constant in time.
\end{itemize}
We notice that we do obtain $v\in Z^1_\tau(\alpha)\otimes X_\sigma$.  Moreover:
\begin{align}
\int \dot u (\alpha \dot v + \beta v)  &= \int \dot u (\beta \overline w + c),\\
& = \int \dot u \beta w. 
\end{align}
This gives:
\begin{align}
\int \dot u (\alpha \dot v + \beta v) \cgeq |u|_{\rmH^{1/2}}^2.
\end{align}

\emph{Next, to bound $v$, we write $v= g - f$ with $g= G_\alpha \ast \overline w$. Then $f$ is the function coinciding with $g$ at extremities and solving $\alpha \ddot f + \beta \dot f = 0$.}

\point Concerning $g$ we have, since $G_\alpha$ has $\rmL^1(\bbR)$-norm one, whatever the Hilbert space $X$:
\begin{equation}
\|g\|_{\rmL^2(X)} \leq \| w \|_{\rmL^2(X)},
\end{equation}
From Corollary \ref{cor:comp} we get:
\begin{align}
|g|_{\rmH^{1/2}} & \cleq |\log(\alpha)|^{1/2} \| w\|_{\rmH^{1/2}},\\
& \cleq  |\log(\alpha)|^{1/2} \| u\|_{\rmH^{1/2}},
\end{align}
From Proposition \ref{prop:varal} we get:
\begin{equation}
(\alpha \int |\dot g|^2)^{1/2} \cleq |\log(\alpha)|^{1/2} \| u\|_{\rmH^{1/2}}.
\end{equation}
We will also use that $g$ is continuous with a bound deduced from Proposition \ref{prop:linfal}:
\begin{align}
\| g\|_{\rmL^\infty} &\leq \| \overline w\|_{\rmL^\infty} \leq  \| w\|_{\rmL^\infty}, \\
& \cleq |\log(\alpha)|^{1/2}( \| w\|_{\rmH^{1/2}}^2    + \alpha \int |\dot w|^2 )^{1/2}, \\
& \cleq |\log(\alpha)|^{1/2}( \| u\|_{\rmH^{1/2}}^2    + \alpha \int |\dot u|^2 )^{1/2}.
\end{align}

\point Concerning $f$ we have the explicit formula:
\begin{equation}
f(t) = g(0)\frac{\exp(-\frac{\beta t}{\alpha}) -  \exp(-\frac{\beta T}{\alpha})}{1 -  \exp(-\frac{\beta T}{\alpha}  )} + g(T)\frac{1 - \exp(-\frac{\beta t}{\alpha})}{1 - \exp(-\frac{\beta T}{\alpha})}.
\end{equation}
We may calculate the Slobodetski seminorm of $f$ on the interval from the expression:
\begin{equation}\label{eq:slobe}
| \exp(-\frac{\beta t}{\alpha}) |^2_{\rmH^{1/2}(0,T)} = \int_0^T|\exp(-\frac{\beta t}{\alpha})|^2 \rmd t \int_0^T \frac{| 1- \exp(-\frac{\beta s}{\alpha})|^2}{s^2} \rmd s.
\end{equation}
The first integral is of order $\alpha$. To estimate the second integral we integrate first from $0$ to $\alpha/\beta$ and then from $\alpha/\beta$ to $T$. The first term is then of order $1/\alpha$ and so is the second one. We conclude that the Slobodetski seminorm (\ref{eq:slobe}) is uniformly bounded as a function of $\alpha$. This gives, for $f$:
\begin{equation}
\|f\|_{\rmH^{1/2}(0,T)}  \cleq \max\{ |g(0)|, |g(T)| \}.
\end{equation}
We also have, from explicit computation:
\begin{equation}
\alpha \int |\dot f|^2 \cleq \max\{ |g(0)|, |g(T)| \}^2.
\end{equation}

\point We now assess a second type of test functions. We denote by $b^\alpha$ the bilinear form defined by:
\begin{equation}
b^\alpha(u,v) = \int \dot u (\alpha \dot v + \beta v).
\end{equation}
Le $\phi$ be a smooth function. We have:
\begin{align}
b^\alpha(u, \phi u) & = \int \alpha \phi |\dot u|^2 + \int (-\alpha \ddot \phi - \beta \dot \phi) \frac{|u|^2}{2}.
\end{align}
Likewise, if $\psi$ is a smooth function, we have:
\begin{align}
b^\alpha(\psi v, v) & = \int \alpha \psi |\dot v|^2 + \int (-\alpha \ddot \psi + \beta \dot \psi)\frac{|v|^2}{2}.
\end{align} 
In the following we let $\phi$ be defined by:
\begin{equation}
\phi (t) = \exp (-t/\kappa).
\end{equation}
We also let $\psi$ be the inverse of $\phi$. We get for a fixed moderate $\kappa$ (say $\kappa = \alpha_0/(2 \beta)$ expressed in terms of the upperbound $\alpha_0$ on $\alpha$):
\begin{align}
-\alpha \ddot \phi - \beta \dot \phi = (- \frac{\alpha}{\kappa^2} + \frac{\beta}{\kappa}) \phi \ceq 1,\\
-\alpha \ddot \psi + \beta \dot \psi = (- \frac{\alpha}{\kappa^2} + \frac{\beta}{\kappa} )\psi \ceq 1.
 \end{align}
By this trick, the test function enables one to dominate also the $\rmL^2$-norm of $u$.
 
\point Now, in the case where $u$ is a discrete trial function, let $v$ be the test function obtained as the upwinded function coinciding with $\phi u$ at vertices. We may consider that $v$ is obtained in two steps, first interpolating $\phi u$ onto piecewise affine functions and then from there to the upwinded functions. The first step is stable, with stable inverse, in the norms of interest by elementary arguments, for $\tau$ small enough. 

We notice the remarkable fact that:
\begin{align}
b^\alpha(u, v) & = b^\alpha(\psi v, v),
\end{align}
because on any $\tau$-interval, $u$ and $\psi v$ coincide at vertices and $v$ is upwinded.
This gives:
\begin{align}
b^\alpha(u, v) & \cgeq  \alpha \int |\dot v|^2 + \int |v|^2,\\
& \cgeq \alpha \int |\dot u |^2 + \int |u|^2,
\end{align} 
We also write:
\begin{align}
\int \gamma u v \geq \int \gamma \phi |u|^2 - \|\gamma\|_{\rmL^\infty} \| u \|_{\rmL^2} \| \phi u - v \|_{\rmL^2} .
\end{align}
We estimate (where the fat dot denotes time derivation):
\begin{align}\label{eq:uvapp}
\int | \phi u - v|^2 &\cleq \tau^2 \int |(\phi u)^\bs|^2,\\
& \cleq \tau^2 (\int |\dot u |^2 + \int |u|^2),
\end{align}
 We deduce that for $\tau$ sufficiently small (independently of $\alpha$) we have:
\begin{equation}\label{eq:sfin}
\int \dot u (\alpha \dot v + \beta v) + \int \gamma u v  \cgeq    \int |u|^2.
\end{equation}

\point We now provide bounds on $v$. Recall that we consider $v$ as obtained in two steps, see the previous point. The stability of the second step is handled by Propositions \ref{prop:asav} and \ref{prop:halfint}. We get:
 \begin{equation}
 | v |_{\rmH^{1/2}} + (\alpha \int |\dot v|^2)^{1/2}  \cleq \| u\|_{\rmL^2} +   |\log(\alpha)|^{1/2} |u|_{\rmH^{1/2}} +  (\alpha \int |\dot u|^2)^{1/2}.
\end{equation} 
In other words we have a stability estimate of order $|\log(\alpha)|^{1/2}$.

\point We now combine the two test functions constructed above, with a parameter $\lambda >0$. Let's call the first one, constructed essentially by projecting the Hilbert transform and convolving, $v_p$, and the second one, constructed by multiplying by $\phi$ and interpolating, $v_i$.
We conclude that the test function:
\begin{equation}\label{eq:testop}
v= v_p + \lambda v_i,
\end{equation}
for $\lambda$ sufficiently larger than $|\log(\alpha)|^{1/2}$ gives and inf-sup condition, deteriorating no faster than $|\log(\alpha)|^{-1}$. To be more precise, the condition on $\lambda$ is that $\lambda \geq C |\log(\alpha)|^{1/2}$, with $C$ sufficiently large (independent of the parameters).

\begin{theorem} Consider the two norms:
\begin{align}
\| u\|_\alpha & = \|u\|_{\rmH^{1/2}_{00}(0,T)} + \alpha^{1/2}\|\dot u\|_{\rmL^2(0,T)},\\
\| v\|_\alpha' & = \|v\|_{\rmH^{1/2}(0,T)} + \alpha^{1/2}\|\dot v\|_{\rmL^2(0,T)}.
\end{align}
Then, for $\tau$ sufficiently small independently of $\alpha$, for any test function $u$, constructing a trial function $v$ by (\ref{eq:testop}) for $\lambda$ sufficiently larger than $|\log(\alpha)|^{1/2}$, we have:
\begin{equation}
\| v\|_\alpha' \cleq |\log(\alpha)| \| u\|_\alpha,
\end{equation}
and:
\begin{equation}
\int \dot u (\alpha \dot v + \beta v) + \int \gamma u v  \cgeq  \| u\|_\alpha^2.
\end{equation}
\end{theorem}

\begin{remark}
We point out that, on an interval $I$, the bilinear form $(u,v) \mapsto \int u \dot v$ is continuous on $\rmH^{1/2}_{00}(I) \times \rmH^{1/2}(I) $. Indeed derivation is continuous from $\rmH^1(I)$ to $\rmL^2(I)$ and from $\rmL^2(I)$ to $\rmH^{-1}(I)$. By interpolation it is continuous from $\rmH^{1/2}(I)$ to the dual of $\rmH^{1/2}_{00}(I)$. 
\end{remark}

\begin{remark}
We mention that with this choice of norms one cannot get an $\alpha$-independent inf-sup condition. Indeed we expect the following type of behaviour, for a solution $u_\alpha$, at the outflow boundary, when $t$ is close to $T$:
\begin{equation}
u_\alpha(t) = 1 - \exp(\frac{\beta (t -T)}{\alpha}).
\end{equation}
Then we have:
\begin{equation}
\int \frac{1}{T-t}|u_\alpha(t)|^2 \ceq |\log(\alpha)|.
\end{equation}
This shows that the $\rmH^{1/2}_{00}$-norm of typical boundary layers blows up, as $\alpha$ tends to $0$,  like $|\log(\alpha)|^{1/2}$.
\end{remark}

\begin{remark}
Ideally one should consider continuity of $(u,v) \mapsto \int u \dot v$ on spaces of the form $\rmH^{1/2}_{0\star}(I) \times \rmH^{1/2}_{\star0}(I) $ where the boundary condition is imposed at inflow on the trial functions and at outflow on the test functions. Explicitely these would be the subspaces of $\rmH^{1/2}(I)$ constituted by functions $u$ such that:
\begin{equation}
\int \rho^{-1} |u|^2 < \infty,
\end{equation}
where $\rho$ is the distance to the part of the boundary where one wants to impose a Dirichlet condition. However the Hilbert transform seems less well adapted to this situation.
\end{remark}

\point We now extend the preceding results to multi-dimension, which essentially introduces crosswind diffusion.

From a trial function $u$  we construct two test functions $v_i$ and $v_p$ by the same method as before. Concerning the second one, we notice the following strengthening of (\ref{eq:sfin}):
\begin{equation}
\int \dot u (\alpha \dot v + \beta v) + \int \gamma u v  \cgeq    \alpha \int |\dot v|^2 +  \int |u|^2.
\end{equation}
We supplement (\ref{eq:uvapp}) with the estimate:
\begin{align}
\alpha \int |\partial (\phi u-v)|^2 & \cleq \alpha \tau^2 \int |\partial (\phi u)^\bs |^2,\\
& \cleq \alpha \frac{\tau^2}{\sigma^2} \int |(\phi u)^\bs |^2,\\
& \cleq \frac{\tau^2}{\sigma^2} \Phi(p)^{-1} \alpha \int |\dot v|^2 + \alpha \frac{\tau^2}{\sigma^2} \int |u|^2,
\end{align}
where $p$ is the P\'eclet number and $\Phi$ was defined in Proposition \ref{prop:phi}. Recall that for large $p$, $\Phi(p)$ behaves like $p$. We deduce that for sufficiently large P\'eclet number we have:
\begin{equation}\label{eq:finest}
\int \dot u (\alpha \dot v + \beta v) + \int \gamma u v + \alpha \int \partial u \cdot \partial v \cgeq \int |u|^2 + \alpha \int |\partial u|^2.
\end{equation}

From there one gets:
\begin{theorem} Consider the two norms:
\begin{align}
\| u\|_\alpha & = \|u\|_{\rmH^{1/2}_{00}(0,T) \otimes \rmL^2(V)} + \alpha^{1/2}\|\nabla u\|_{\rmL^2(U)},\\
\| v\|_\alpha' & = \|v\|_{\rmH^{1/2}(0,T) \otimes \rmL^2(V)} + \alpha^{1/2}\|\nabla v\|_{\rmL^2(U)}.
\end{align}
Then, for $\tau$ sufficiently small, independently of $\alpha$, and for P\'eclet numbers bounded below by a sufficiently large number, for any test function $u$, constructing a trial function $v$ by (\ref{eq:testop}) for $\lambda$ sufficiently larger than $|\log(\alpha)|^{1/2}$,  we have:
\begin{equation}
\| v\|_\alpha' \cleq |\log(\alpha)| \| u\|_\alpha,
\end{equation}
and:
\begin{equation}
\int \beta \dot u  v + \int \gamma u v  + \alpha \int \nabla u \cdot \nabla v \cgeq  \| u\|_\alpha^2.
\end{equation}
\end{theorem}

\begin{remark}
There are other circumstances, in addition to large P\'eclet number, that guarantee an estimate of the type (\ref{eq:finest}). For instance $\tau$ sufficiently smaller than $\sigma$. If we allow for anisotropic diffusion, if the crosswind diffusion is sufficiently smaller than the streamline diffusion, this will also be sufficient.
\end{remark}

\section*{Acknowledgement}
We thank Martin Werner Licht and Espen Sande for numerous corrections on the first version.

Kind hospitality of \'Ecole Normale Sup\'erieure, Paris, from October 2015 to February 2016, is gratefully acknowledged. 

This research was supported by the European Research Council through the FP7-IDEAS-ERC Starting Grant scheme, project 278011 STUCCOFIELDS.

\bibliography{../../Bibliography/alexandria,../../Bibliography/newalexandria,../../Bibliography/mybibliography}{}
\bibliographystyle{plain}


\newpage

\appendix 
\section{Alternative stability proof}

\emph{The first version of the paper contained the following alternative proof of stability for the convection diffusion problem. It is more cumbersome than the one provided above. We include it here mainly because we have corrected some typos.} 

\paragraph{Some preliminary results.}

\begin{proposition}\label{prop:linf}
There exists  $C > 0$ such that for all $\tau > 0$ and all $ u$ which are continuous and $\tau$-piecewise affine:
\begin{equation}
 \| u \|_{\rmL^{\infty}} \leq C |\log (\tau)|^{1/2} \| u \|_{\rmH^{1/2}}.
\end{equation}
\end{proposition}
\begin{proof}
We have, for small $s >0$:
\begin{align}
\|u\|_{\rmL^\infty} & \cleq \| \calF u \|_{\rmL^1},\\
&\cleq \int (1 + |\xi|^2)^{- \frac{1 + s}{4}} (1 + |\xi|^2)^{\frac{1 + s}{4}} |\calF u (\xi)| \rmd \xi,\\
& \cleq \left(\int (1 + |\xi|^2)^{- \frac{1 + s}{2}} \rmd \xi\right)^{1/2}\left (\int (1 + |\xi|^2)^{\frac{1 + s}{2}} |\calF u(\xi)|^2 \rmd \xi\right)^{1/2} ,\\
& \cleq \frac{1}{s^{1/2}} \|u \|_{\rmH^{(1+s)/2}}.
\end{align} 
Then we use an inverse inequality to obtain:
\begin{align}
\|u\|_{\rmL^\infty} & \cleq \frac{\tau^{-s/2} }{s^{1/2}}   \|u \|_{\rmH^{1/2}}.
\end{align}
Finally  we let:
\begin{equation}
s = 1/ |\log(\tau)|.
\end{equation}
This gives the estimate.
\end{proof}

\begin{proposition}\label{prop:linfh}
On any given interval, there exists a $C > 0$ such that for all $\tau > 0$ and all functions $ u$ which are continuous, $\tau$-piecewise affine and $0$ at the extremities:
\begin{equation}
 \| \calH u \|_{\rmL^{\infty}} \leq C |\log (\tau)| \| u \|_{\rmL^\infty}.
\end{equation}
\end{proposition}
\begin{proof} We use that for large $p$, the Hilbert transform is continuous from $\rmL^p$ to $\rmL^p$, with norm of order $p$.
We write, for large $p < \infty$:
\begin{align}
\| \calH u \|_{\rmL^{\infty}} & \cleq \| \calH u \|_{\rmL^{p}}^{1 - 1/p} \| \calH \dot u \|_{\rmL^{p}}^{1/p},\\
& \cleq p \| u  \|_{\rmL^{p}}^{1 - 1/p} \| \dot u \|_{\rmL^{p}}^{1/p},\\
& \cleq p \tau^{-1/p} \| u \|_{\rmL^{\infty}}.
\end{align}
Then we choose:
\begin{equation}
p = |\log (\tau)|.
\end{equation}
This concludes the proof.
\end{proof}
Combining the two previous propositions we get:
\begin{corollary}
On any given interval, there exists a $C > 0$ such that for all $\tau > 0$ and all functions $ u$ which are continuous, $\tau$-piecewise affine and $0$ at the extremities:
\begin{equation}
 \| \calH u \|_{\rmL^{\infty}} \leq C |\log (\tau)|^{3/2} \| u \|_{\rmH^{1/2}}.
\end{equation}
\end{corollary}

The following is another error estimate for an operator producing upwinded functions.
\begin{proposition}\label{prop:err}
For each $\epsilon >0$ there exists $C > 0$ such that for all $\alpha \leq \tau /C$ and all $u$ which are $\tau$-piecewise constant, we have
\begin{equation}
\| u - G_\alpha \ast u \|_{\rmL^2} \leq \epsilon \| u\|_{\rmL^2}.
\end{equation}
\end{proposition}

\begin{proof}
We first fix $\tau= 1$.

\noindent For $\delta > 0$ we decompose $G_\alpha$ as follows:
\begin{equation}
G_\alpha = G_\alpha^\delta + (G_\alpha - G_\alpha^\delta),
\end{equation}
with:
\begin{equation}
G_\alpha^\delta(s) = \left \{ \begin{array}{ll}
\frac{\beta}{\alpha} \exp( \frac{- \beta s}{\alpha}) &\textrm{ for } 0 \leq s \leq \delta,\\
0 &\textrm{ for } s < 0 \textrm{ or } s > \delta.
\end{array} \right.
\end{equation}
We have:
\begin{equation}
\int_0^{\delta}G_\alpha^\delta  \rmd s = 1 - \exp (\frac{-\beta \delta}{\alpha}).
\end{equation}
Fix $\delta \in ]0, 1[$. As $\alpha$ tends to $ 0$ the above number tends to $1$.

Choose $\epsilon > 0$. For small enough $\alpha$ we have, for any $u$ that is constant on $]-1, 0[$ and $]0, 1[$: 
\begin{equation}
\| u - G_\alpha^\delta \ast u \|_{\rmL^2(0, 1)}  \leq \frac{\epsilon}{4} \| u \|_{\rmL^2(-1, 1)}.
\end{equation}
Therefore, for $u\in \rmL^2(\bbR)$ which is constant on each interval $]k, k+1[$ for $k \in \bbZ$:
\begin{align}
\| u - G_\alpha^\delta \ast u \|_{\rmL^2}^2 & \leq  \frac{\epsilon^2}{16} \sum_{k \in \bbZ} \| u \|_{\rmL^2(k-1, k+1)},\\
& \leq \frac{ \epsilon^2 }{8} \|u\|^2_{\rmL^2}.
\end{align}
For small enough $\alpha$ we also have:
\begin{equation}
\| G_\alpha - G^\delta_\alpha \|_{\rmL^1} = \exp (\frac{-\beta \delta}{\alpha}) \leq \frac{\epsilon}{4}.
\end{equation}
Therefore:
 \begin{equation}
\| u - G_\alpha \ast u \|_{\rmL^2} \leq \epsilon \, \frac{ 2^{1/2} + 1}{4} \|u \|_{\rmL^2}. 
\end{equation}
This gives the result for $\tau = 1$. One concludes by scaling.
\end{proof}

\paragraph{Problem setup.} For a function $u$ defined on a domain $U=]0,T[ \times V$, derivation along the first axis will be denoted $u \mapsto \dot u$, and derivation along the remaining axes (in $V$) will be denoted $u \mapsto  \partial_V u$.

In this section we take $\gamma= 1$. The variational form of equation (\ref{eq:convdiff}) can be written:
\begin{equation}\label{eq:cdvar}
\int \langle \dot u, \alpha \dot v + \beta v \rangle + \int a^\alpha (u, v) = \int \langle f, v \rangle.
\end{equation}
Here, integration is on $]0, T[$, and for functions on $V$ we denote:
\begin{equation}
\langle u, v \rangle = \int_V u v.
\end{equation}
Moreover $a^\alpha$ denotes the bilinear map defined on functions on $V$ by:
\begin{equation}
a^\alpha(u,v) = \int_V u v + \alpha \int_V \partial_V u \cdot \partial_V v.
\end{equation}

We let $X_\sigma$ denote some standard finite element space of functions on $V$, such as $Q_{1...1}$ with respect to a product mesh of width $\sigma$, or $P_1$ with respect to a simplicial mesh.

Let $Z^0_\tau$ denote the space of $\tau$-piecewise constant functions on $[0,T]$. Also let $Z^1_\tau$ denote the space of continuous $\tau$-piecewise affine ones, which are $0$ at the extremities of the interval. Finally let $Z^1_\tau(\alpha)$ denote the space of continuous functions which are $\tau$-piecewise of the upwinded form (\ref{eq:cpe}), which are also $0$ at the extremities.

As already indicated in \S \ref{sec:setup} we solve (\ref{eq:cdvar}) by a Petrov Galerkin method, where the trial space is $Z^1_\tau \otimes X_\sigma$ and the test space is $Z^1_\tau(\alpha) \otimes X_\sigma$.
  
\paragraph{Decomposition in steps and terms.}

Given $u\in Z^1_\tau \otimes X_\sigma$, we construct a quasi-optimal test function for (\ref{eq:cdvar}) in $Z^1_\tau(\alpha) \otimes X_\sigma$ in several steps:
\begin{itemize}
\item $v_0 = \calH u + \lambda u$, with $\lambda \geq 1$ to be determined,
\item $v_1 = \overline v_0 \in Z^0_\tau \otimes X_\sigma$ is the projection of $v_0$ onto $\tau$-piecewise constants,
\item $v_2$ solves $\alpha \dot v_2 + \beta v_2 = \beta v_1$, with $v_2(0)=0$,
\item $v_3$ is defined by putting:
\begin{align}
v_3 & = v_2 \textrm{ on } [0, T - \tau],\\
\alpha \ddot v_3 + \beta \dot v_3 & = 0 \textrm{ on } ]T - \tau, T], \textrm{ with } v_3(T)=0 .
\end{align}
\end{itemize} 
This last $v_3$  is in $Z^1_\tau(\alpha) \otimes X_\sigma$ and will be our candidate for an optimal test function. Our first task is to show how relevant norms of $v_3$ can be controlled. We then write:
\begin{equation}
\int \langle \dot u, \alpha \dot v_3 + \beta v_3 \rangle + \int a^\alpha (u, v_3) = I_1 + I_2 + I_3 + I_4,
\end{equation}
with:
\begin{align}
I_1 & = \int_{[0, T - \tau]} \langle \dot u, \beta v_1 \rangle,\\
I_2 & = \int_{[T - \tau, T]} \langle \dot u, \alpha \dot v_3 + \beta v_3 \rangle,\\
I_3 & = \int_{[0, T - \tau]} a^\alpha(u,v_2),\\
I_4 & = \int_{[T - \tau, T]} a^\alpha(u,v_3).
\end{align} 
We estimate the four terms successively. The overall plan is to show that $I_1$ and $I_3$ are big and together dominate the norm squared of $u$, whereas the terms $I_2$ and $I_4$ will be shown not to deteriorate this estimate.

We have four parameters: $\alpha, \tau, \sigma$ and $\lambda$. All our constants are independent of these parameters. Notice also that we will let $\lambda$ vary, contrary to the theory provided for parabolic problems, where it was just chosen big enough.

\paragraph{Controlling $v_3$.} We first remark that $v_3$ is not too big. More precisely we have the following estimates. By Corollary  \ref{cor:comp} we have:
\begin{equation}
\| v_2 \|_{\rmH^{1/2}(O)} \cleq |\log(\alpha)|^{1/2}\lambda  \| u \|_{\rmH^{1/2}(O)}.
\end{equation}
We also have:
\begin{equation}
\| v_2 \|_{\rmL^2(X^\alpha)} \cleq \lambda  \| u \|_{\rmL^2(X^\alpha)}.
\end{equation}

To estimate $v_3$ we use the explicit formula, for $t \in [T-\tau, T]$:
\begin{equation}\label{eq:expl}
v_3(t) = \frac{\exp(\frac{\beta(T - t)}{\alpha}) -1}{\exp(\frac{\beta \tau}{\alpha}) -1} v_2 (T- \tau).
\end{equation}
We can deduce the following formula on $[T- \tau, T]$:
\begin{equation}\label{eq:abv3}
\alpha \dot v_3 + \beta v_3 = \frac{-\beta}{\exp(\frac{\beta \tau}{\alpha}) -1} v_2(T- \tau).
\end{equation}
We remark that for the characteristic function of $[T-\tau, T]$ we have:
\begin{align}
\| \chi_{[T-\tau, T]} \|_{\rmH^{1/2}_w(O)}^2  & = \| \chi_{[T-\tau, T]} \|_{\rmL^2(O)}^2 + | \chi_{[T-\tau, T]} |_{\rmH^{1/2}_w(O)}^2,\\
& \leq \tau + 1.
\end{align}
Using the notation of the proof of Proposition \ref{prop:diffnorm} we introduce $v = v_2 - v_3$ and write:
\begin{align}
v & = G_\alpha \ast\frac{1}{\beta}( \alpha \dot v + \beta v),\\
 & = G_\alpha \ast (v_1 +  \frac{1}{ \exp(\frac{\beta \tau}{\alpha}) - 1 }v_2 (T- \tau)) \chi_{[X-\tau, X]}.
\end{align}
We deduce, using Propositions \ref{prop:wh} and \ref{prop:diffnorm} that, for small $\epsilon$:
\begin{equation}
\|v\|_{\rmH^{1/2}(O)} \cleq \frac{1}{\epsilon^{1/2} \alpha^{\epsilon}} ( \| v_1\|_{\rmL^{\infty}(O)} + \| v_2 \|_{\rmL^{\infty}(O)}).
\end{equation}
We let $\epsilon =  1 / |\log (\alpha)|$ and combine with Propositions \ref{prop:linf} and \ref{prop:linfh} to deduce:
\begin{equation}
\|v\|_{\rmH^{1/2}(O)} \cleq  |\log (\alpha)|^{1/2} |\log(\tau)|^{1/2} ( |\log(\tau)| + \lambda) \| u\|_{\rmH^{1/2}(O)}.
\end{equation}

We conclude:
\begin{proposition} We have the following estimates:

\begin{equation}
\|v_3\|_{\rmH^{1/2}(O)} \cleq |\log (\alpha)|^{1/2} |\log(\tau)|^{1/2} ( |\log(\tau)| + \lambda) \| u\|_{\rmH^{1/2}(O)},
\end{equation}
and:
\begin{equation}
\| v_3 \|_{\rmL^2(X^\alpha)} \cleq \lambda  \| u \|_{\rmL^2(X^\alpha)}.
\end{equation}
\end{proposition}

Our next task is to get the lower bound on $I_1 + I_2 + I_3 + I_4$.
\paragraph{Estimating $I_1$.}

\begin{align}
I_1 & = \int_{[0, T - \tau]} \beta \langle \dot u , \calH u + \lambda u \rangle,\\
& = \beta |u|_{\rmH^{1/2}(O)}^2 - \beta \int_{[T - \tau, T]} \langle \dot u , \calH u \rangle +  \frac{\lambda \beta}{2} \|u(T - \tau)\|_0^2.
\end{align}
In this equation we remark that:
\begin{align}
|\int_{[T - \tau, T]} \langle \dot u , \calH u \rangle | & = \int_{[T - \tau, T]} |\langle \frac{u(T- \tau)}{\tau} , \calH u \rangle |,\\
& \leq \frac{1}{2 \epsilon} \|u(T- \tau)\|_0^2 + \frac{\epsilon}{2 \tau} \int_{[T - \tau, T]} \|\calH u\|_0^2.
\end{align}
Moreover, Propositions \ref{prop:linf}, \ref{prop:linfh} give a constant $C_1$ so that: 
\begin{align} 
\frac{1}{\tau} \int_{[T - \tau, T]} \|\calH u\|_0^2 & \leq \| \calH u \|_{\rmL^{\infty}(O)}^2,\\
& \leq C_1 |\log (\tau)|^{3} \| u \|_{\rmH^{1/2}(O)}^2.
\end{align}
Choosing: 
\begin{equation}
\epsilon = \frac{1}{C_1 C_2 |\log(\tau)|^{3}},
\end{equation}
we get:
\begin{align} 
|\int_{[T - \tau, T]} \langle \dot u , \calH u \rangle | & \leq \frac{C_1C_2|\log(\tau)|^{3} }{2}  \|u(T-\tau)\|_0^2 + \frac{1}{2 C_2} \| u\|_{\rmH^{1/2}(O)}^2.
\end{align}

All in all, we get:
\begin{align}
I_1/\beta  \geq |u|_{\rmH^{1/2}(O)}^2 - \frac{1}{2C_2} \|u\|_{\rmH^{1/2}(O)}^2 + \frac{\lambda - C_1C_2|\log(\tau)|^{3}}{2} \|u(T- \tau)\|_0^2.
\end{align}

In the following we suppose that $\lambda$ satisfies:
\begin{equation}\label{eq:lambdab}
\lambda \geq 2 C_1C_2 |\log(\tau)|^{3}. 
\end{equation}

\paragraph{Estimating $I_2$.}

Integration by parts, using (\ref{eq:abv3}) gives:
\begin{equation}
I_2 = \frac{\beta}{\exp(\frac{\beta \tau}{\alpha}) -1} \langle u(T-\tau), v_2(T-\tau) \rangle. 
\end{equation}
Using also Propositions \ref{prop:linf}, \ref{prop:linfh}, we deduce:
\begin{align}
|I_2| & \cleq \frac{1}{\exp(\frac{\beta \tau}{\alpha}) -1} \| u(T-\tau)\|_0 \|v_2 (T- \tau) \|_0,\\
      & \cleq \frac{1}{\exp(\frac{\beta \tau}{\alpha}) -1} \| u(T-\tau)\|_0 |\log(\tau)|^{3/2}\lambda  \|u\|_{\rmH^{1/2}(O)},\\
      & \cleq \frac{\lambda}{\exp(\frac{\beta \tau}{\alpha}) -1}  ( |\log(\tau)|^{3} \| u(T-\tau)\|_0^2  +  \|u\|_{\rmH^{1/2}(O)}^2 ).
\end{align}

Suppose we have an inequality:
\begin{equation}\label{eq:alphab}
\alpha \leq \frac{\beta \tau}{|\log(\tau)|},
\end{equation}
then we have:
\begin{equation}
\frac{1}{\exp(\frac{\beta \tau}{\alpha}) -1} \leq \frac{\tau}{1 - \tau}. 
\end{equation}
In the following we suppose also that we have an inequality:
\begin{equation}\label{eq:lambdau}
\lambda \leq \frac{1}{C \tau},
\end{equation}
for some large enough $C$.

Then we deduce that $I_2$ does not deteriorate the estimate for $I_1$. That is, assuming (\ref{eq:lambdab}), (\ref{eq:lambdau}) and (\ref{eq:alphab}), and arbitrarily large $C$, we get an estimate:
\begin{equation}
I_1 + I_2 \cgeq |u|_{\rmH^{1/2}(O)}^2 - \frac{1}{C} \|u\|_{\rmL^2(O)}^2 + \lambda \|u(T- \tau)\|_0^2.
\end{equation}

\paragraph{Estimating $I_3$.}
We write:
\begin{align}
I_3 = \int_{[0, T - \tau]} a^\alpha(u, v_1) + \int_{[0, T - \tau]} a^\alpha(u,v_2 - v_1).
\end{align}
For the first term on the right hand side:
\begin{align}
\int_{[0, T - \tau]} a^\alpha(u, v_1)  &= \lambda \int_{[0, T - \tau]} a^\alpha(\overline u,\overline u) + \int_{[0, T - \tau]} a^\alpha(\overline u,\calH u).
\end{align}
Here we remark that:
\begin{equation}
\int_{[0, T - \tau]} a^\alpha(\overline u,\overline u) \cgeq  \|\overline u\|_{\rmL^{2}(X^\alpha)}^2,
\end{equation}
and that:
\begin{equation}
|\int_{[0, T - \tau]} a^\alpha(\overline u,\calH u)| \cleq \| \overline u\|_{\rmL^{2}(X^\alpha)} \| u\|_{\rmL^{2}(X^\alpha)}.
\end{equation}
For the second term we have:
\begin{equation}
| \int_{[0, T - \tau]} a^\alpha(u,v_2 - v_1)| \cleq \|u\|_{\rmL^2(X^\alpha)} \| v_2 - v_1\|_{\rmL^2(X^\alpha)}.
\end{equation}
Now, according to Proposition \ref{prop:err}, under assumption (\ref{eq:alphab}), we get estimates, for arbitrarily large $C$:
\begin{equation}
\| v_2 - v_1\|_{\rmL^2(X^\alpha)} \leq \frac{1}{C} \|v_1\|_{\rmL^2(X^\alpha)}.
\end{equation}
We combine this with the estimate:
\begin{equation}
\|v_1\|_{\rmL^2(X^\alpha)} \cleq  \| u \|_{\rmL^2(X^\alpha)} + \lambda \| \overline u \|_{\rmL^2(X^\alpha)}.
\end{equation}

All in all we deduce that for arbitrarily large $C$ we may get an estimate:
\begin{equation}
I_3 \cgeq \lambda \|\overline u\|_{\rmL^{2}(X^\alpha)}^2 - \frac{1}{C} \|u\|_{\rmL^{2}(X^\alpha)}^2.
\end{equation}

\paragraph{Estimating $I_4$.}
We use the explicit formula (\ref{eq:expl}). We compute:
\begin{equation}
\int_{[T-\tau, T]} \frac{\exp(\frac{\beta(T - t)}{\alpha}) -1}{\exp(\frac{\beta \tau}{\alpha}) -1}\rmd t = \frac{\alpha}{\beta} - \frac{\tau}{\exp(\frac{\beta \tau}{\alpha}) -1}.
\end{equation}

We therefore have:
\begin{align}
|I_4| & \cleq \alpha \|u(T -\tau)\|_{X^\alpha} \|v_3(T -\tau)\|_{X^\alpha}.
\end{align}
Here we substitute:
\begin{align}
\|v_3(T -\tau)\|_{X^\alpha} & \cleq \| v_1\|_{\rmL^\infty(X^\alpha)}, \\
& \cleq \|\calH u\|_{\rmL^\infty(X^\alpha)} + \lambda \|\overline u\|_{\rmL^\infty(X^\alpha)} ,\\
& \cleq |\log(\tau)|  \|u\|_{\rmL^{\infty}(X^\alpha)} +  \lambda \|\overline u\|_{\rmL^\infty(X^\alpha)} ,\\
& \cleq \tau^{-1/2} \left(|\log(\tau)|  \|u\|_{\rmL^{2}(X^\alpha)} + \lambda \| \overline u\|_{\rmL^{2}(X^\alpha)}\right).
\end{align}
Therefore:
\begin{align}
| I_4 |&\cleq   \frac{\alpha}{   \tau} |\log(\tau)|  \|u\|_{\rmL^{2}(X^\alpha)}^2      +        \frac{\alpha^{3/2}}{\sigma \tau^{1/2}} \lambda \|u(T-\tau)\|_0 \| \overline u\|_{\rmL^{2}(X^\alpha)}.
\end{align}
Now we strengthen hypothesis (\ref{eq:alphab}) to the following one:
\begin{equation}\label{eq:alphabs}
\alpha \leq \frac{\min\{\tau, \sigma\} }{C|\log(\tau)|},
\end{equation}
for a large constant $C$. We deduce:
\begin{align}
|I_4 |&\cleq \frac{1}{C}  \|u\|_{\rmL^{2}(X^\alpha)}^2 + \frac{\lambda}{C^{3/2}} (\|u(T-\tau)\|_0^2 +  \| \overline u\|_{\rmL^{2}(X^\alpha)}^2).\label{eq:i4e}
\end{align}
For a large enough $C$, $|I_4|$ is dominated by (any fraction of) $I_1 + I_3$.

\paragraph{Combination of estimates.}
All in all we get:
\begin{align}
I_1 + I_2 + I_3 + I_4 & \cgeq |u|_{\rmH^{1/2}(O)}^2  + \lambda \|u(T- \tau)\|_0^2 + \lambda \|\overline u\|_{\rmL^{2}(X^\alpha)}^2  \\
&\qquad  - \frac{1}{C} \|u\|_{\rmL^2(O)}^2 - \frac{1}{C} \|u\|_{\rmL^{2}(X^\alpha)}^2.
\end{align}
Combined with Lemma \ref{lem:equiv} we deduce in particular:
\begin{align}
I_1 + I_2 + I_3 + I_4 & \cgeq |u|_{\rmH^{1/2}(O)}^2  + \| u\|_{\rmL^{2}(X^\alpha)}^2.
\end{align}
Recalling (\ref{eq:lambdab}) we also have:
\begin{align}
|v_3|_{\rmH^{1/2}(O)} + \| v_3 \|_{\rmL^{2}(X^\alpha)} & \cleq C(\alpha,\tau)   ( |u|_{\rmH^{1/2}(O)}  + \| u\|_{\rmL^{2}(X^\alpha)}),\\
C(\alpha, \tau) & = |\log(\tau)|^{7/2}|\log(\alpha)|^{1/2}.
\end{align}

Summing up we get:
\begin{theorem}
The convection diffusion equation (\ref{eq:convdiff}), discretized by the above Petrov-Galerkin method on quasi uniform grids, assuming a condition (\ref{eq:alphabs}), satisfies a uniform discrete inf-sup condition up to logarithmic terms in $\alpha$ and $\tau$. 
\end{theorem}

\end{document}